\documentclass{article}
\usepackage{fullpage}

\usepackage[utf8]{inputenc}
\usepackage{amsthm,amssymb,amsmath,bbm}
\usepackage{caption,subcaption}
\usepackage{epsfig,graphicx,graphics,color,tikz,tikz-cd,tcolorbox}
\usepackage{enumerate,enumitem}
\usepackage[sort]{cite}
\usepackage{hyperref}
\usepackage{todonotes}
\usepackage{mathtools}
\usepackage[normalem]{ulem}
\hypersetup{
    colorlinks=true,       % false: boxed links; true: colored links
    linkcolor=blue,        % color of internal links (change box color with linkbordercolor)
    citecolor=magenta,         % color of links to bibliography
    filecolor=magenta,     % color of file links
    urlcolor=cyan,         % color of external links
    linktoc=all
}

\def\final{0}  % set this to 1 to get a comment-free version
\def\iflong{\iffalse}
\ifnum\final=0  %namely if we allow comments in the output
\newcommand{\kristof}[1]{{\color{red}[{ \textbf{Kristóf:}  #1}]\marginpar{\color{red}*}}}
\newcommand{\marci}[1]{{\color{blue}[{ \textbf{Marci:}  #1}]\marginpar{\color{blue}*}}}
\newcommand{\llaci}[1]{{\color{orange}[{ \textbf{LLaci:}  #1}]\marginpar{\color{orange}*}}}
\newcommand{\tlaci}[1]{{\color{purple}[{ \textbf{TLaci:}  #1}]\marginpar{\color{purple}*}}}
\else % in this case [final=1] we don't want any comments to show
\newcommand{\kristof}[1]{}
\newcommand{\marci}[1]{}
\newcommand{\llaci}[1]{}
\newcommand{\tlaci}[1]{}
\fi

\usepackage{stackengine}
\stackMath

\theoremstyle{plain}
\newtheorem{thm}{Theorem}[section]
\newtheorem{lem}[thm]{Lemma}
\newtheorem{cor}[thm]{Corollary}

\theoremstyle{definition}

\newtheorem{ex}[thm]{Example}

\newtheorem{rem}[thm]{Remark}

\newcommand{\R}{\mathbb{R}}
\newcommand{\Z}{\mathbb{Z}}
\newcommand{\N}{\mathbb{N}}

\def\Gb{\mathbf{G}}
\def\Nbb{\mathbb{N}}

\newcommand{\cA}{\mathcal{A}}
\newcommand{\cB}{\mathcal{B}}

\newcommand{\cS}{\mathcal{S}}
\newcommand{\cC}{\mathcal{C}}

\newcommand{\cF}{\mathcal{F}}
\newcommand{\FF}{\mathcal{F}}
\newcommand{\cH}{\mathcal{H}}
\newcommand{\cK}{\mathcal{K}}
\newcommand{\cI}{\mathcal{I}}

\newcommand{\cP}{\mathcal{P}}
\newcommand{\PP}{\mathcal{P}}

\newcommand{\cY}{\mathcal{Y}}
\newcommand{\cZ}{\mathcal{Z}}

\newcommand{\rto}{\rightarrowtail}

\newcommand{\fg}{\varphi}
\newcommand{\eps}{\varepsilon}
\newcommand{\Haus}{{\text{\rm Haus}}}
\newcommand{\hto}{\stackrel{\Haus}{\longrightarrow}}
\def\one{{\mathbbm1}}
\def\rk{\text{\rm rk}}

\newcommand{\psih}{\psi_{\cH}}
\newcommand{\psib}{\psi_{\cB}}
\newcommand{\psik}{\psi_{\cK}}

\newcommand{\dH}[1]{d_{\Haus_{2^{#1}}}} % Erre a távolságra valami jobb jelölés jó lenne.

\usepackage{titling}

\title{Quotient-convergence of Submodular Setfunctions}
\thanksmarkseries{alph}
\author{
Kristóf Bérczi\thanks{MTA-ELTE Matroid Optimization Research Group and HUN-REN–ELTE Egerváry Research Group, Department of Operations Research, Eötvös Loránd University, and HUN-REN Alfréd Rényi Institute of Mathematics, Budapest, Hungary. Email: \texttt{kristof.berczi@ttk.elte.hu}.}
\and
Márton Borbényi\thanks{Department of Computer Science, Eötvös Loránd University  and HUN-REN Alfréd Rényi Institute of Mathematics, Budapest, Hungary. Email: \texttt{marton.borbenyi@ttk.elte.hu}.}
\and
László Lovász\thanks{HUN-REN Alfréd Rényi Institute of Mathematics, Budapest, Hungary. Email: \texttt{laszlo.lovasz@ttk.elte.hu}.}
\and
László Márton Tóth\thanks{HUN-REN Alfréd Rényi Institute of Mathematics, Budapest, Hungary. Email: \texttt{toth.laszlo.marton@renyi.hu}.}
}
\date{}

\begin{document}
\maketitle
\tableofcontents

%%%%%%%%%%%%%%%%%%%%%%%%%%%%%%%%
 \newpage
%%%%%%%%%%%%%%%%%%%%%%%%%%%%%%%%

\begin{abstract} 
We introduce the concept of quotient-convergence for sequences of submodular set functions, providing, among others, a new framework for the study of convergence of matroids through their rank functions. Extending the limit theory of bounded degree graphs, which analyzes graph sequences via neighborhood sampling, we address the challenge posed by the absence of a neighborhood concept in matroids. We show that any bounded set function can be approximated by a sequence of finite set functions that quotient-converges to it. In addition, we explicitly construct such sequences for increasing, submodular, and upper continuous set functions, and prove the completeness of the space under quotient-convergence.

\medskip


\end{abstract}
%%%%%%%%%%%%%%%%%%%%%%%%%%%%%%%%

%%%%%%%%%%%%%%%%
\section{Introduction}
\label{sec:intro}
%%%%%%%%%%%%%%%%

The theory of graph limits provides tools to analyze sequences of graphs, offering insights into the structural similarities of large graphs through analytic methods. This analytic viewpoint helps to understand complex networks and their applications in various fields of mathematics and computer science, such as statistical physics, network theory and probability theory. For graph sequences, convergence of a growing sequence of graphs  
can be defined in two complementary was: in terms of the distributions of small subgraphs sampled randomly from the members of the sequence (left-convergence), and in terms of homomorphisms into small graphs (right-convergence). Under mild conditions, these two notions turn out equivalent (see~\cite{bensch2001recurrence,bckl2013leftandright} for bounded-degree graphs and~\cite{BCLSV1, BCLSV2} for dense graphs). 

Once convergence is defined, there is a natural question to assign a limit object for convergent graph sequences. Depending on the specifics of the sequence, such limit objects can be described as unimodular random locally finite graphs, graphons, graphings, scaling limits, and so on; see~\cite{lovasz2012large} for details. Once a limit object is defined in a given class of structures, the ``soficity'' question arises: is every member of this class a limit object of finite structures? This question can be quite easy (e.g.\ for dense graphs) or a long-standing open problem (e.g. the Aldous--Lyons Conjecture for bounded-degree graphs).

Matroids capture fundamental graph features such as spanning trees, matchings, connectivity and planarity, and thus provide a unified framework for studying graph-related problems; see~\cite{Welsh,oxley2011matroid,Schrijver}. The question of a limit theory for matroids has been raised by several authors. Left-convergence, based on local properties, has been studied in~\cite{kardovs2017first}; this concept focuses on small (finite) subsets. We take the other route and introduce right-convergence, better reflecting macroscopic properties of matroids. One possible axiomatization of matroids is through the rank function, which extends the standard notion of rank of a set of vectors. The most important property of the rank function $\rk$ is that it is submodular: 
\begin{equation*}
  \rk(X)+\rk(Y) \geq \rk(X \cap Y)+\rk(X \cup Y)  
\end{equation*}
for all subsets $X,Y$ of the ground set. 

Limits of matroid rank functions and other important submodular setfunctions from combinatorics are submodular setfunctions on infinite sets, which leads to an analytic study of submodularity. Independently of the combinatorial applications, this research direction was initiated as early as in the 1950s by Choquet~\cite{choquet1954theory}, starting a line of research of ``nonlinear integral'' and related topics; see~\cite{Denn, kechris2012classical}. The study of connections between the analytic and combinatorial theories of submodularity has recently been initiated~\cite{lovasz2023submodular} and a connection with local-global convergence of bounded-degree graphs has been established in~\cite{lovasz2023matroid} and our parallel paper~\cite{bblt2024graphing}. 

\medskip

In this paper, we introduce a form of right-convergence, called \emph{quotient-convergence}, of a sequence of  setfunctions; this provides, in particular, a notion of convergence of matroids through their rank functions (Section~\ref{sec:convergence}). We show that every quotient-convergent sequence of finite increasing submodular setfunctions converges to an increasing submodular setfunction on a Borel space (Section~\ref{sec:complete}), where we call a setfunction {\it finite} if its domain is finite. This implies by standard arguments the completeness of the space of increasing submodular setfunctions under quotient-convergence. We settle the soficity problem by proving that for any bounded setfunction, there exists a sequence of finite setfunctions that quotient-converges to it, and we also give an explicit construction of such a sequence when the setfunction is increasing, submodular and continuous from above (Section~\ref{sec:finite}). 

The limit object is not unique; our construction provides a limiting submodular setfunction that is continuous from below. In fact, the most natural limit object is a submodular setfunction on the clopen subsets of the Cantor set (if we allow setfunctions defined on set-algebras rather than sigma-algebras). The relationship between these representations of the limit is not fully understood.

%%%%%%%%%%%%%%%%
\section{Quotient-convergence}
\label{sec:convergence}
%%%%%%%%%%%%%%%%

We need some definitions. Let $\cF$ be a family of subsets of the set $J$. A {\it setfunction} on $(J,\cF)$ is a function $\varphi\colon\cF\to\R$. A {\it lattice family} is a family $\cF$ of subsets of a set $J$ closed under union and intersection. A setfunction $\varphi$ on a lattice family $(J,\cF)$ is {\it submodular}, if 
\begin{equation*}
  \varphi(X)+\varphi(Y) \geq \varphi(X \cap Y)+\varphi(X \cup Y)  
\end{equation*}
for all $X,Y\in\FF$. Submodularity is known to be equivalent to $\varphi(A\cup X)-\varphi(A\cup Y)\geq \varphi(X)-\varphi(Y)$ for all $A,X,Y\in\cB$ with $X\subseteq Y$, also referred to as the {\it diminishing returns} property. In this paper, we will assume that $\fg(\emptyset)=0$.

We need several different smoothness conditions on a setfunction $\varphi$. If $\cF$ is closed under countable union, then we say that $\varphi$ is lower continuous (or continuous from below), if for every sequence $X_1\subseteq X_2\subseteq\ldots$ of sets in $\cF$, if $\varphi(\cup_nX_n)=\lim_n\varphi(X_n)$. {\it Upper continuous} (or continuous from above) is defined analogously. We say that an increasing setfunction $\varphi$ defined on the compact subsets of a topological space is {\it continuous from the right} if for all $\varepsilon>0$ and $K\in\cK$, there exists an open set $U$ such that $\varphi(K')<\varphi(K)+\varepsilon$ for all $K'\subseteq U$, $K'\in\cK$.

\subsection{The quotient profile}

We denote by $\N$ the \emph{set of positive integers}. For $k\in\N$, we define the {\it $k$-quotient set} $Q_k(\fg)$ as the set of all quotients of $\fg$ on $[k]$. In other words, $Q_k(\fg)$ is the set of all setfunctions $\fg\circ F^{-1}$, where $F\colon J\to[k]$ is a measurable map. The set $Q_k(\fg)$ is a subset of $\R^{2^k}$, with the two coordinates $\fg\circ F^{-1}(\emptyset)=0$ and $\fg\circ F^{-1}(J)=\fg(J)$ being fixed. One can think of $F$ as a measurable partition of $J$ into $k$ (possibly empty) parts, and $\varphi\circ F^{-1}$ is a quotient of $\varphi$ by substituting each partition class by a single element. Therefore, we also use the notation $\varphi/\cP$ for denoting the quotient of $\varphi$ corresponding to a partition $\cP$ of $J$. Note that $Q_k(\fg)$ is invariant under the permutations of $[k]$. We call the sequence $(Q_k(\fg)\colon k\in\N)$ the {\it quotient profile} of $\fg$. The sequence of closures $(\overline{Q}_k(\fg)\colon k\in\N)$ is the {\it closed quotient profile} of $\fg$. Often we only need to speak about the {\it quotient set} $Q(\fg)=\cup_k Q_k(\fg)$ of all finite quotients of $\fg$.

The quotient profile contains a lot of information about the setfunction. Let us give some examples.

\begin{ex}[Finitary properties]\label{EXA:FIN-CONV}
Many properties can be expressed in terms of the function values on finite set-algebras contained in $(J,\cB)$: for example, increasing, submodular, infinite-alternating, finitely additive, $0$-$1$ valued. Equivalently, such a property holds for a setfunction $\varphi$ if and only if it holds for every quotient $\varphi/\PP$, where $\PP$ is a partition into a finite number of measurable sets. Such properties are called {\it finitary} in~\cite{bgils2024decomposition}. 
\end{ex}

\begin{ex}[Ideals]\label{EXA:IDEALS}
Every $0$-$1$ valued increasing submodular setfunction on $(J,\cB)$ is of the form $\varphi(X)=\one(X\notin\cI)$, where $\cI$ is a set ideal. Every quotient set $Q_k(\varphi)$ consists of all setfunctions $\psi_{k,A}\colon 2^{[k]}\to\{0,1\}$ where $\psi_{k,A}(X)=\one(A\cap X\not=\emptyset)$ for some nonempty subset $A\subseteq[k]$. Furthermore, among these setfunctions, those defined by maximal ideals are exactly the ones that are finitely additive. This is a finitary property, and so this means that all finite quotients are defined by maximal ideals, i.e., they are of the form $\delta_x$ for a single element $x$. However, the profile does not indicate whether the ideal is principal or not.
\end{ex}

\begin{ex}[Binary matroids]\label{EXA:MATROIDS}
A {\it matroid} $M=(J,\rk)$ is defined by its finite {\it ground set} $J$ and its {\it rank function} $\rk\colon2^J\to\Z$ that satisfies the {\it rank axioms}: (R1) $\rk(\emptyset)=0$, (R2) $X\subseteq Y\Rightarrow \rk(X)\leq \rk(Y)$, (R3) $\rk(X)\leq |X|$, and (R4) $\rk(X)+\rk(Y)\geq \rk(X\cap Y)+\rk(X\cup Y)$. Observe that (R2) asserts that the rank functions is increasing, while (R4) asserts the rank function being submodular. A subset $X\subseteq J$ is {\it independent} if $|X|=\rk(X)$. The \emph{rank} of the matroid is $\rk(J)$.

Let $M=(J,\rk)$ be a matroid and $A\subseteq J$. The \emph{deletion of $A$}, also called the {\it restriction to $J\setminus A$} results in the matroid $M|A=(J\setminus A,\rk')$ with rank function $\rk'(X)=\rk(X)$ for $X\subseteq J\setminus A$. The \emph{contraction of $A$} results in a matroid $M/A=(J\setminus A,\rk'')$ with rank function $\rk''(X)=\rk(X\cup A)-\rk(A)$ for $X\subseteq J\setminus A$. A matroid that can be obtained from $M$ by a sequence of restrictions and contractions is called a \emph{minor} of $M$. 

The quotient profile of the rank function of a matroid carries a lot of structural information about the matroid itself, one particularly interesting example being the representability over the two-element field. A matroid is called \emph{binary} if its elements can be identified with the columns of a $(0,1)$-matrix in such a way that a subset of elements is independent if and only if the corresponding columns are linearly independent in GF(2). Tutte~\cite{tutte1958homotopy} showed that binary matroids admit a forbidden minor characterization: a matroid is binary if and only if it has no minor isomorphic to $U_{4,2}$.

Let $M=(J,\rk)$ be a matroid. Then, by looking at $Q_6(\rk)$, one can decide if the matroid is binary or not. Indeed, by Tutte's characterization, $M$ is binary if and only if its ground set does not admit a partition into six parts $J=J_1\cup \dots\cup J_6$ such that $J_i$ contains a single element for $i=1,2,3,4$, and $\rk\big((\cup_{i\in I} J_i)\cup J_5\big)-\rk(J_5)=\min\{|I|,2\}$ for each $I\subseteq[4]$. In other words, $M$ is binary if and only if $Q_6(\rk)$ does not contain a setfunction $f\colon[6]\to\Z$ with $f(I\cup\{5\})-f(\{5\})=\min\{|I|,2\}$ for each $I\subseteq [4]$.

A property of matroids is {\it minor-closed}, if it is inherited by minors. Every minor-closed property can be defined by excluding a list of minors (quite often, but not always, a finite list). Many fundamental properties of matroids are minor-closed (e.g., linearly representable, graphic, cographic, algebraic). The argument above shows that every minor-closed property described by a finite list of excluded minors can be read off from an appropriate quotient set $Q_k(\rk)$.
\end{ex}

For two setfunctions $\varphi,\varphi'$, we say that $\varphi$ is a {\it lift} of $\varphi'$ if $\varphi'$ is a quotient of $\varphi$. Let $A$ be a family of finite setfunctions. We say that $A$ is {\it quotient-closed}, if every quotient of a function in $A$ is also contained in $A$. We say that $A$ has the {\it common lift property}, if any finite set of functions in $A$ has a common lift in $A$.

We close this section with a simple lemma.

\begin{lem}\label{lem:quotient-prof-lift}
The quotient profile of any setfunction $\varphi$ is quotient-closed and has the common lift property. 
\end{lem}
\begin{proof}
Trivially, every quotient of a quotient of $\fg$ is a quotient of $\fg$. Given a finite set of quotients, every quotient is defined by a corresponding partition, and the quotient defined by the common refinement of these partitions is a common lift.
\end{proof}

The proof shows that if $\varphi_i\in Q_{k_i}(\varphi)$, then $\varphi_1,\dots,\varphi_m$ have a common lift in $Q_{k_1\cdots k_m}(\fg)$.

\subsection{Convergence}

Let $\varphi_n$ be a setfunction on a set-algebra $(J_n,\cB_n)$ for $n\in \N$. We say that the sequence $(\fg_n\colon n\in\N)$ is {\it quotient-convergent} if for every $k\in\N$, the quotient sets $(Q_k(\fg_n)\colon n\in\N)$ form a Cauchy sequence in the Hausdorff distance $\dH{k}$ on the $2^k$-dimensional Euclidean space. Let $\hto$ denote convergence in this distance. We say that $(\fg_n\colon n\in\N)$ {\it quotient-converges to} $\fg$ for some setfunction $\fg$ if the sequence $(Q_k(\fg_n)\colon n\in\N)$ converges to $Q_k(\fg)$ in the Hausdorff distance, and denote this by $\fg_n \rightarrowtail \fg$. We can combine these Hausdorff distances $\dH{k}$ into a single pseudometric as follows. For two setfunctions $\fg_1$ and $\fg_2$, define
\begin{align}\label{EQ:HAUS-DIST}
d(\fg_1,\fg_2) = \sum_{k=1}^\infty 2^{-k} \dH{k}(Q_k(\fg_1),Q_k(\fg_2)).
\end{align} 
Then quotient-convergence of a sequence $\fg_n$ means that it is a Cauchy sequence, and $\fg_n \rightarrowtail \fg$ means convergence in this pseudometric. It is easy to see that a quotient-convergent sequence of setfunctions is uniformly bounded. 

An important example of quotient-convergent sequences of setfunctions is the following. A {\it tower} is a sequence $(\fg_n\colon n\in\Nbb)$ of setfunctions on finite set-algebras $(J_n,2^{J_n})$, such that $\fg_n$ is a quotient of $\fg_{n+1}$ for every $n$. We assume for convenience that $J_n=[v_n]$ for some $v_n\in\Nbb$, and that $\fg_n(\emptyset)=0$\footnote{We could also allow $J_n$ to be a Polish space, and the setfunction $\varphi_n$ to be defined on the Borel subsets. Then, with a few additional assumptions, many results in the following section could be extended, but the treatment would become too technical and not to the point of this paper.}. Trivially, $\fg_n(J_n)=\fg_1(J_1)$ for every $n\ge 1$.

\begin{lem}\label{lem:tower-conv}
Every tower is quotient-convergent.
\end{lem}
\begin{proof}
Let $(\fg_n\colon n\in\Nbb)$ be a tower. We may assume that $0\le\fg_n\le 1$. By Lemma~\ref{lem:quotient-prof-lift}, we have
\begin{equation}\label{EQ:TOWER1}
\overline{Q}_k(\fg_n) \subseteq \overline{Q}_k(\overline{Q}_{v_n}(\fg_{n+1})) = \overline{Q}_k(\fg_{n+1}).
\end{equation}
So for every $k$, the sequence $(\overline{Q}_k(\fg_n)\colon n\in\Nbb)$ is increasing, which implies trivially that it is convergent in the Hausdorff distance.
\end{proof}

The following lemma is a certain converse to Lemma~\ref{lem:quotient-prof-lift}.

\begin{lem}\label{lem:lift-quotient-prof}
Let $S_k$ be a nonempty set of setfunctions on $[k]$ for $k\in\N$, and assume that $S=\cup_k S_k$ is quotient-closed and has the common lift property. Then there are setfunctions $\fg_k\in S_k$ such that $(\fg_n\colon n\in\N)$ is a tower, and $\cup_n Q_k(\fg_n)$ is a dense subset of $S_k$.
\end{lem}

\begin{proof}
    For every $k,n\in\N$, let $A_k^{n}$ be a finite $(1/k)$-net in $S_k$, and set $B_n=\cup_{k=1}^n A_k^{n}$. Thus, $B_n$ is a finite, fine net in the first $n$ sets $S_k$. Let $\psi_0\in S_0$, and let $\psi_n\in S$ be a common lift of $\psi_{n-1}$ and all setfunctions in $B_n$. Then $(\psi_1,\psi_2,\dots)$ is a tower, and since $S$ is quotient-closed, $\cup_n Q_k(\fg_n)\subseteq S_k$. Also, $A_{k,n}\in\cup_n Q_k(\fg_n)$, which implies that $\cup_n Q_k(\fg_n)$ is dense in $S_k$.
    
    The sequence $(\psi_n\colon n\in\N)$ does not contain an element from every $S_k$, but it is easy to ``pad'' it to a sequence $(\fg_n\colon n\in \N)$ as in the lemma. In fact, if $\psi_n\in S_{k_n}$, then we can define $\fg_{k_n}=\psi_n$, and define $\fg_k$ for $k=k_{n+1}-1, k_{n+1}-2, \dots, k_n+1$ recursively by merging two elements of $[k+1]$ that are mapped onto one in the quotient map $\psi_{n+1}\to\psi_n$.
\end{proof}

We continue with more examples.

\begin{ex}[Uniform and sparse paving matroids]\label{EXA:PAVING} 
For $n\in\N$ and $0\le r\le n$, we define a {\it uniform matroid} $U_{n,r}$ on $[n]$ by its rank function $\rk(X)=\min\{|X|,r\}$. Let us normalize to obtain the submodular function $\rho_{n,r}(X)=\rk(X)/n=\min\{|X|/n,r/n\}$. Then $\rho_{n,r}$ is the truncation of the uniform probability measure $\pi_n$ on $[n]$ by $r/n$. Hence for every partition $\cP$ of $[n]$, the submodular setfunction $\rho_{n,r}/\cP$ is the truncation of $\pi_n/\cP$ by $r/n$. In particular, for any $k\in\N$, $Q_k(\rho_{n,r})$ can be obtained from $Q_k(\pi_n)$ by truncating all its elements by $r/n$. Let $1\le r_n<n$ for $n\in\N$ where $r_n\sim c\cdot n$ for some fixed constant $0<c<1$. Let $\lambda$ be the Lebesgue measure on $[0,1]$, and let $\lambda_c=\min\{\lambda,c\}$. Then it is not difficult to check that for every $k\in\N$, $Q_k(\rho_{n,r_n})\to Q_k(\lambda_c)$ in Hausdorff distance, and hence $\rho_{n,r_n}\rightarrowtail \lambda_c$.

A much larger class of matroids can be obtained by the following slight modification of uniform matroids. Let $\cH\subseteq \binom{[n]}{r}$ where $n>r$, and assume that $|A\cap B|\le r-2$ for any two sets $A,B\in\cH$. Define
\[
\rk_\cH(X)=
  \begin{cases}
    |X| & \text{if $|X|\le r-1$}, \\
    r-1 & \text{if $X\in\cH$}, \\
    r & \text{otherwise}.
  \end{cases}
\]
Then $M_\cH=([n],\rk)$ is a matroid of rank $r$, called a {\it sparse paving matroid} defined by $\cH$. Specifically, choosing $\cH=\emptyset$ gives the uniform matroid $U_{n,r}$. Let us normalize to obtain the submodular function
$\rho_\cH(X)=\rk_\cH(X)/n$. Then
\begin{equation*}\label{EQ:PAVE-UNIFORM}
\rho_{n,r}(X)-\frac1n\le\rho_\cH(X)\le \rho_{n,r}(X).
\end{equation*}
This implies that if $r_n\sim c\cdot n$ and $\cH_n\subseteq \binom{[n]}{r_n}$ as above, then $\dH{k}\big(Q_k(\rho_{\cH_n}),Q_k(\rho_{n,r_n})\big) \to 0$,
and hence $\rho_{\cH_n}\rightarrowtail \lambda_c$.
\end{ex}

\begin{ex}[Dense convergence and graphons]\label{EXA:GRAPHON-CONV}
Let $G_n=(V_n,E_n)$ $(n\in\Nbb)$ be a sequence of graphs converging to a graphon $W\colon [0,1]^2\to [0,1]$. For $X\subseteq V(G_n)$, let $e_n(X)$ denote the number of edges in $G_n$ connecting $X$ to $V(G_n)\setminus X$. Then $\rho_n(X) = e_n(X)/|V_n|^2$ is a finite submodular setfunction with $0\le \rho_n\le 1$. Let us define 
\[
\rho(X)=\int_{X\times ([0,1]^2\setminus X)} W(x,y)\,dx\,dy.
\]
Then $\rho$ is a submodular setfunction on the Borel sets in $[0,1]$, and $\rho_n \rightarrowtail \rho$. The proof of this fact is somewhat involved, and it uses the characterization of the convergence $G_n\to W$ in terms of the $q$-quotients $\cS_q(W)$ (similar to our definition of quotient-convergence) in~\cite{BCLSV2,lovasz2012large}.
\end{ex}

\begin{ex}[Sparse convergence and graphings]\label{EXA:GRAPHING-CONV}
It is shown in~\cite{bblt2024graphing} that if a sequence of bounded graphs $(G_n\colon n\in\Nbb)$ converges to a graphing $\Gb$ in the local-global sense, then their normalized matroid rank functions quotient-converge to the rank function of $\Gb$, defined in~\cite{lovasz2023matroid}. It is also shown that local convergence is not enough; we refer to~\cite{lovasz2012large} and~\cite{bblt2024graphing} for the definitions.
\end{ex}

%%%%%%%%%%%%%%%%
\section{Construction of the limit object}
\label{sec:complete}
%%%%%%%%%%%%%%%%

This section is dedicated to the proof that the space of submodular setfunctions (endowed with the pseudometric~\ref{EQ:HAUS-DIST}) is complete. Formally, we prove the following.

\begin{thm}\label{thm:complete}
    Let $(\varphi_n\colon n\in N)$ be a quotient-convergent sequence of increasing submodular setfunctions over any set-algebras. Then there exists an increasing submodular setfunction $\varphi\colon\cB(C)\to\R$ defined on the Borel sets of the Cantor set such that $\varphi_n\rightarrowtail\varphi$. 
\end{thm}

The proof will consist of two parts. In Section~\ref{sec:clopen}, we describe how to construct a submodular setfunction on the algebra of clopen subsets of the Cantor set. Then the function is extended to all Borel subsets using a classical result of Choquet (Section~\ref{sec:borel}). Theorem~\ref{thm:complete} will follow easily by combining these results.

%%%%%%%%%%%%%%%%
\subsection{Limit object on the algebra of clopen subsets}
\label{sec:clopen}
%%%%%%%%%%%%%%%%

In this section, we construct a submodular setfunction on the clopen subsets of the Cantor set, which is a limit object for a quotient-convergent sequence of submodular setfunctions.

\begin{lem}\label{LEM:TOWER-LIMIT}
For every tower $(\fg_n\colon n\in\N)$ of finite setfunctions, there is a setfunction $\fg$ on the algebra $\cH$ of clopen sets of the Cantor set and a refining sequence of partitions $(\PP_n\colon n\in\N)$ of $(C,\cH)$, such that $\fg_n=\fg/\PP_n$. Furthermore, $\fg_n\rto\fg$.
\end{lem}

\begin{proof}
Let $\fg_n$ be defined on the finite Boolean algebra $(J_n,2^{J_n})$. By the definition of quotients and of towers, there are maps $f_n\colon J_{n+1}\to J_n$ such that $\fg_n=\fg_{n+1}\circ f_n^{-1}$.

Consider the inverse limit of the system of setfunctions $\fg_n$ and maps $f_n$. Explicitly, we consider the set $X$ of all sequences $(x_n\colon n\in\N)$ such that $x_n\in J_n$ and $x_n=f_n(x_{n+1})$. We define the map $\Theta_n\colon J\to J_n$ as the projection onto the $n^{\text{th}}$ coordinate:
$\Theta_n(x_1,x_2,\dots) = x_n$. Then $\Theta_n=f_n\circ\Theta_{n+1}$. This set can be endowed with a metric, in which the distance of $(x_1,x_2,\dots)$ and $(y_1,y_2,\dots)$ is $2^{-r}$, where $r$ is the length of their longest common prefix. This topology of $X$ is totally disconnected, metrizable and compact.

For every $n$, the sets $\{\Theta_n^{-1}(u)\mid u\in J_n\}$ form a finite partition $\PP_n$ of $J$, and they generate a (finite) set-algebra $\cA_n$. Clearly, $(\PP_n\colon n\in\N)$ is a refining sequence of partitions into sets generating $\cH$, $\cA_n\subseteq\cA_{n+1}$, and $\cH=\cup_n\cH_n$ is the set-algebra $\cH$ of clopen subsets of $X$. We define $\fg(X)=\fg_n(\Theta_n(X))$ for $X\in\cA_n$. It is easy to see that this value is independent of $n$, as soon as $X\in\cA_n$. Trivially, $\fg$ is a setfunction on $\cH$, and $\fg_n=\fg\circ\Theta_n^{-1}$. So $\fg_n$ is a quotient of $\fg$. In terms of the partitions, this means that $\fg/\PP_n$ is isomorphic to $\fg_n$.

To prove the last assertion, notice that $Q_k(\fg_n)\subseteq Q_k(\fg)$. It suffices to show that $\cup_k Q_k(\fg_n)=Q_k(\fg)$. Indeed, if $\alpha\in Q_k(\fg)$, then $\alpha=\fg/\PP$ for a finite $k$-partition into sets in $\cH$. These partition classes belong to $\cA_n$ for a sufficiently large $n$. This implies that the sets $\{\Theta_n(U)\mid U\in\cP\}$ are disjoint, and form a partition $\PP'$ of $[n]$. Then $\fg/\cP =\fg_n/\cP'$, showing that $\alpha\in Q_k(\fg_n)$.

Right now, we have constructed a limit object on some totally disconnected compact metric space $X$. To obtain a representation on the Cantor set, we use the well-known fact that $X$ can be embedded homomorphically as a closed subset $Y$ of the Cantor set $C$. Defining $\fg'(U)=\fg(Y\cap U)$, we get a representation on the clopen subsets of the Cantor set. 
\end{proof}

\begin{cor}\label{cor:sofic_weak}
For every quotient-convergent sequence $(\fg_n\colon n\in\N)$ of finite setfunctions there is a setfunction $\fg$ on the algebra $\cH$ of clopen sets of the Cantor set such that $\fg_n\rto\fg$.
\end{cor}
\begin{proof}
As remarked, a quotient-convergent sequence is uniformly bounded, and so we may assume that $-1\le\fg_n\le 1$. By definition, the sets $Q_k(\fg_n)$ $(n\in\Nbb)$ form a Cauchy sequence in the Hausdorff distance for every $k$. Since the space of compact subsets of compact set is compact in this pseudometric, there are compact sets $S_k$ such that $\overline{Q}_k(\fg_n)\to S_k$. From the fact that every $Q=\cup_k Q_k(\fg_n)$ is quotient-closed, it follows that $S=\cup_kS_k$ is quotient closed. From the fact that $Q$ has the common lift property, along with the remark after Lemma~\ref{lem:quotient-prof-lift}, it follows that $S$ has the common lift property. Hence, by Lemma~\ref{lem:lift-quotient-prof}, there are setfunctions $\psi_k\in S_k$ such that $(\psi_n\colon n\in\N)$ is a tower, and $\cup_n \overline{Q}_k(\psi_n)=S_k$, which is equivalent to saying that $\overline{Q}_k(\psi_n)\to S_k$. By Lemma~\ref{LEM:TOWER-LIMIT}, there is a setfunction $\fg$ on the algebra $\cH$ such that $\psi_n\rto \fg$, or $\overline{Q}_k(\psi_n)\to\overline{Q}_k(\fg)$, which means that $\overline{Q}_k(\fg)=S_k$. Then $\overline{Q}_k(\fg_n)\to \overline{Q}_k(\fg)$, which means that $\fg_n\rto\fg$.
\end{proof}

Let us point out that a very similar argument gives that if the setfunctions $\fg_n$ are increasing and/or submodular, then so is their limit function on the clopen subsets of the Cantor set.

%%%%%%%%%%%%%%%%
\subsection{Extension to all Borel sets}
\label{sec:borel}
%%%%%%%%%%%%%%%%

In this section we extend the limit setfunction constructed in the previous section from the clopen subsets to all Borel subsets of the Canter set, so that the closed quotient profile is preserved. From now on, we need the condition that the setfunction is increasing and submodular, but these properties will also be preserved. We use two results of Choquet~\cite{choquet1954theory}; see also~\cite[Chapter 30]{kechris2012classical} for a more recent introduction. A setfunction on the compact subsets of a topological space is called an {\it alternating capacity of order $2$} if it is increasing, submodular, and continuous from the right. One of the main results in~\cite{choquet1954theory} is the following theorem. 

\begin{thm}[Choquet]\label{thm:choquet}
Let $\varphi_{\cK}$ be an alternating capacity of order $2$ on the compact subsets of a Polish space. For any subset $S$, let 
    $$\varphi_*(S)=\sup_{\substack{K\subseteq S\\K\text{ compact}}}\varphi_{\cK}(K)$$ 
be the inner capacity, and let 
    $$\varphi^*(S)=\inf_{\substack{S\subseteq U\\U\text{ open}}}\varphi_*(U)$$ be the outer capacity. 
Then the setfunction $\varphi^*$ is increasing and submodular. Furthermore, if $B$ is Borel (or analytic) set, then $\varphi_*(B)=\varphi^*(B)$. 
\end{thm}

We also need a simple but very useful technical lemma from~\cite{choquet1954theory}. 

\begin{lem}[Choquet]\label{LEM:CHOQ}
Let $\phi$ be an increasing submodular set function on a lattice family 
$(J,\cF)$. For $i\in[n]$, let $A_i\subseteq B_i$ be pairs of sets in $\cF$. Then
\[
\phi(\cup_{i=1}^n B_i)-\phi(\cup_{i=1}^n A_i)\le \sum_{i=1}^n(\phi(A_i)-\phi(B_i)).
\]
\end{lem}

Turning to the extension from the clopen sets to Borel sets, this will be carried out in two steps: first, we extend the setfunction from the family $\cH$ of clopen subsets to the family $\cK$ of compact subsets, then to the family $\cB$ of Borel sets. We will denote these setfunctions by $\psih$, $\psik$ and $\psib$. We start with $\psih=\fg$, where $\fg$ is the function provided by Corollary~\ref{cor:sofic_weak}. 

First, we extend $\fg$ to all compact sets by setting
\begin{equation}
\psik(K)\coloneqq\inf_{\substack{K\subseteq H\\ H\in\cH}}\psih(H).\label{eq:psi}    
\end{equation}
Since $\psih$ is increasing, $\psik$ is indeed an extension of $\psih$. 
    
\begin{lem}\label{lem:k}
The function $\psik$ is increasing, submodular and continuous from the right. 
\end{lem}
\begin{proof}
If $K\subseteq K'$, then we have $\{H\in\cH|\ K'\subseteq H\}\subseteq\{H\in\cH|\ K\subseteq H\}$, so the infimum in~\eqref{eq:psi} cannot be larger for $K$ than for $K'$. This proves that $\psi$ is increasing.

Let $K_1,K_2$ be compacts sets, $\varepsilon>0$, and $H_1,H_2$ be clopen sets such that $K_1\subseteq H_1$, $K_2\subseteq H_2$ and $\psik(H_1)-\psik(K_1)<\varepsilon$, $\psik(H_2)-\psik(K_2)<\varepsilon$. Then 
\begin{align*}
\psik(K_1)+\psik(K_2)
{}&{}> \psik(H_1)+\psik(H_2)-2\varepsilon\\
{}&{}\ge \psik(H_1\cap H_2)+\psik(H_1\cup H_2)-2\varepsilon\\
{}&{}\ge \psik(K_1\cap K_2)+\psik(K_1\cup K_2)-2\varepsilon,
\end{align*}
where the second inequality holds by submodulartiy of $\psih$ on the clopen sets, while the third inequality holds by $K_1\cap K_2\subseteq H_1\cap H_2$ and $K_1\cup K_2\subseteq H_1\cup H_2$. Since the inequality is true for any $\varepsilon>0$, taking $\varepsilon\to0$ proves the submodularity of $\psi$.

Finally, let $\varepsilon>0$ and $K\in\cK$. Then, by the definition of the extension, there exists not only an open but even a clopen set $U$ with $\psik(K')<\psik(K)+\varepsilon$ for all $K'\subseteq U$, $K'\in\cC$. This proves continuity from the right.
\end{proof}

We use Theorem~\ref{thm:choquet} to extend our function from $\cK$ to $\cB$. By Lemma~\ref{lem:k}, $\psik$ is an alternating capacity of order $2$, so we can apply Choquet's theorem to extend the function to $\cB$ by defining $\psib(B)\coloneqq\psi_*(B)=\psi^*(B)$.

\begin{rem}
    There is another way to extend the function to all Borel sets. First, let us define $\overline{\psi}(S)\coloneqq \inf_{S\subseteq H\in\cH}\psih(H)$. Note that the value only depends on the closure of the set, so we can also express it as $\overline{\psi}(S)=\psik(\overline{S})$. Second, we make this function continuous from below by setting $\psib'(S)\coloneqq \overline{\psi}_{\bullet}(S) =\inf_{\mathcal{S}}\lim_{n\to\infty}\overline{\psi}(S_n)$, where $\mathcal S=\{(S_n\colon n\in\N)\mid|\ S_n\text{ is ascending},\ \cup_n S_n=S  \}$. Note that this approach is similar to the definition of the Lebesgue measure, where first one defines the Jordan measure for clopen subsets, then extend it to all subsets as the Jordan outer measure, and then makes it continuous from below to obtain the Lebesgue outer measure. One can prove that the function thus obtained is the same as the construction by Choquet, see~\cite[Chapter 6]{lovasz2023submodular} for further details.
\end{rem}

Since $\psib$ extends $\psih$, is clear that $Q_k(\psih)\subseteq Q_k(\psib)$. To prove that their closures are equal, we show that each Borel partition can be approximated by a clopen partition.

\begin{lem}\label{lem:clopen_approx}
    Let $K_1,\dots,K_q$ be pairwise disjoint compact subsets and $U_1,\dots,U_q$, open subsets of the Cantor set $C$ such that $K_i\subseteq U_i$ and $\cup_{i=1}^q U_i=C$. Then for every $i\in[q]$ there exist clopen sets $H_i$ such that $K_i\subseteq H_i\subseteq U_i$ and $\{H_1,\dots,H_q\}$ forms a partition of $C$. 
\end{lem}
\begin{proof}
    For all $x\in C$, let us choose a clopen subset $U_x$ containing $x$ as follows. If $x\in K_i$, then $U_x\subseteq U_i$ and $U_x\cap K_j=\emptyset$ for all $j\not=i$. If $x\in (\cup_{i=1}^q K_i)^c$, then $U_x\cap(\cup_{i=1}^q K_i)=\emptyset$ and there exists an index $i\in[q]$ with $U_x\subseteq U_i$. As the clopen subsets form a basis for the topology, we can find a $U_x$ with the given property for all $x\in C$. Since $C\subseteq \cup_{x\in C} U_x$ and $C$ is compact, there exist a finite list of elements $x_1,...,x_n$ such that $C\subseteq \cup_{i=1}^n U_{x_i}. $  
    
    Let $V_1,V_2,...,V_\ell$ be the atoms in the set-algebra generated by the clopen sets $U_{x_1},U_{x_2},\dots, U_{x_n}$. Consider an arbitrary index $j\in[\ell]$. If $V_j$ intersects $K_i$, then it cannot intersect $K_k$ for $k\neq i$. Indeed, $V_j\subseteq U_{x_t}$ for some $t\in[n]$, and $U_{x_t}$ intersects at most one of the compact sets $K_1,\dots,K_q$. Furthermore, we have $V_j\subseteq U_{x_t}\subseteq U_i$. In this case, we add $V_j$ to the set $H_i$. If $V_j$ does not intersect any of the compact sets $K_i$, then it is still covered by $U_{x_\ell}$ for some $\ell\in[n]$, and so it is covered by $U_i$ for some $i\in[q]$. In this case, we choose such an index $i$ and add $V_j$ to the set $H_i$. 
    
    By construction, we clearly have $H_i\subseteq U_i$ and also $K_i\subseteq L^{n}_i\subseteq H_i$, where $L^{n}_i$ consists of the $n^{\text{th}}$ level clopen sets which intersects $K_i$.
\end{proof}

\begin{lem}\label{lem:clopen_approx2}
    Let $F\colon C\to[q]$ be a Borel measurable function. Then for all $\varepsilon>0$, there exists a continuous map $G\colon C\to[q]$ such that $d(\psib\circ F^{-1},\psib\circ G^{-1})<\varepsilon$. 
\end{lem}
\begin{proof}
    Let $B_i=F^{-1}(\{i\})$ and let $\delta>0$. Then, by Theorem~\ref{thm:choquet}, there exists $K_i\subseteq B_i\subseteq U_i$ such that $K_i$ is compact, $U_i$ is open, $\psib(B_i)-\psib(K_i)<\delta$ and $\psib(U_i)-\psib(B_i)<\delta$. Clearly, the sets $K_1,\dots,K_q$ are pairwise disjoint while $\cup_{i=1}^q U_i$ covers the whole space. By Lemma~\ref{lem:clopen_approx}, there exists a partition $C=H_1\cup\dots\cup H_q$ into clopen sets such that $K_i\subseteq H_i\subseteq U_i$. Let $G$ be defined as $G(i)\coloneqq \{j\mid i\in H_j\}$.

    For any set $A\subseteq [q]$, let $B_A=\cup_{i\in A}B_i$, $K_A=\cup_{i\in A}K_i$, $H_A=\cup_{i\in A}H_i$ and $U_A=\cup_{i\in A}U_i$. Using Lemma~\ref{LEM:CHOQ}, we have $\psib(B_A)-\psib(K_A)<\delta|A|$ and $\psib(U_A)-\psib(B_A)<\delta|A|$, implying $|\psib(B_A)-\psib(H_A)|<\delta|A|.$ Therefore, $\psib\circ F^{-1}$ and $\psib\circ G^{-1}$ can differ by at most $\delta q$ in each coordinate. By choosing $\delta$ small enough, the statements follows.
\end{proof}

We get the following corollary.

\begin{cor}\label{cor:h2b}
    The setfunctions $\psih$ and $\psib$ have the same quotient profile. Therefore, if $\fg_n\rto\psi_H$ then $\fg_n\rightarrowtail \psib$.
\end{cor}

\begin{proof}
    Clearly, $Q_k(\psih)\subseteq Q_k(\psib)$. By Lemma~\ref{lem:clopen_approx2}, it follows that $Q_k(\psib)\subseteq \overline{Q_k(\psih)}$, thus $\overline{Q_k(\psih)}=\overline{Q_k(\psib)}$.
\end{proof}

The proof of Theorem~\ref{thm:complete} follows by Corollaries~\ref{cor:sofic_weak} and~\ref{cor:h2b}.

\begin{cor}
    The space of increasing normalized submodular setfunctions on Borel spaces, endowed with the pseudometric~\eqref{EQ:HAUS-DIST}, is compact. 
\end{cor}

\begin{proof}
    As the metric comes from the embedding into the compact metric space $\prod_{k\in\N} \cK([0,1]^{2^k})$, where $\cK(A)$ is the compact subsets of $A$ endowed with the Hausdorff metric, the space of increasing normalized submodular setfunctions on Borel spaces is totally bounded. Theorem~\ref{thm:complete} states that this space is complete. From these observations, it follows that the space in question is compact.
\end{proof}

%%%%%%%%%%%%%%%%
\section{Finite approximations}
\label{sec:finite}
%%%%%%%%%%%%%%%%

In this section we answer the {\it soficity question}, which is the natural question to characterize the limit objects of convergent sequences. In our case the answer is easy: Every bounded setfunction defined on a set-algebra is a limit.

\begin{thm}\label{thm:sofic_strong}
    Let $\varphi$ be a bounded setfunction defined on a set-algebra. Then there exists a sequence of finite quotients $\varphi_n$ of $\varphi$ such that $\varphi_n\rto\varphi$.
\end{thm}

It is easy to see that we may assume that $\varphi_n$ is defined on the subsets of $[n]$.

\begin{proof}
The construction will be similar to part of the proof of Theorem~\ref{thm:complete}. For every $k,n\in\N$, let $A_k^{n}$ be a finite $(1/n)$-net in $Q_k(\varphi)$, and set $B_n=\cup_{k=1}^n A_k^{n}$. So $B_n$ intersects each of the first $n$ quotient sets in a finite $(1/n)$-net. Let $\varphi_n$ be a common lift of the setfunctions in $B_n$. We claim that $\varphi_n$ quotient-converges to $\varphi$ as $n$ tends to infinity. The functions $\varphi_n$ are finite and are quotients of $\varphi$, so $Q_\ell(\varphi_n)\subseteq Q_\ell(\varphi)$. For all $\ell\le n$, $Q_\ell(\varphi_n)$ contains a $(1/n)$-net of $Q_\ell(\varphi)$, thus $\dH{\ell} (Q_\ell(\varphi_n),Q_\ell(\varphi))\le 1/n$. Therefore, $Q_\ell(\varphi_n)\to Q_\ell(\varphi)$ for all $\ell\in\N$, implying $\varphi_n\rto\varphi$.
\end{proof}

Under further restrictions on $\fg$, we can give an explicit construction for a sequence of finite setfunctions quotient-converging to $\fg$. We need the following technical lemma.

\begin{lem}\label{LEM:AA-APPROX}
Let $\fg$ be an increasing submodular setfunction on a sigma-algebra $(J,\cB)$, continuous from above, with $\fg(\emptyset)=0$. Let $\cA\subseteq\cB$ be a set-algebra generating $\cB$. Then, for every $X\in\cB$ and every $\eps>0$ there exists $Y\in\cA$ such that $\fg(X\triangle Y) <\eps$.
\end{lem}

\begin{proof}
Let $\cS=\{X\in\cB\mid \text{for every $\eps>0$ there exists $Y\in\cA$ with $\fg(X\triangle Y)<\eps$}\}$. We claim that $\cS$ is a sigma-algebra. We clearly have $\cA\subseteq\cS$, hence $\emptyset\in\cS$. Since for the symmetric difference we have $(J\setminus X)\triangle (J\setminus Y)=X\triangle Y$, $\cS$ is closed under complementation. For $X_1,X_2\in\cS$ and $\eps>0$, there exist sets $Y_1,Y_2\in\cA$ such that $\fg(X_i\triangle Y_i)<\eps/2$. As $(X_1\cup X_2)\triangle(Y_1\cup Y_2)\subseteq(X_1\triangle Y_1)\cup(X_2\triangle Y_2)$, the monotonicity and submodularity of $\fg$ gives 
\begin{equation*}
\fg\bigl((X_1\cup X_2)\triangle(Y_1\cup Y_2)\bigr)\le \fg(X_1\triangle Y_1)+\fg(X_2\triangle Y_2)\le\eps.
\end{equation*}
Since $Y_1\cup Y_2\in\cA$, this implies that $X_1\cup X_2\in\cS$. Thus, by $\cS$ being closed under complementation, we get that $X_1\cap X_2\in\cS$.

To show that $\cS$ is also closed under countable union, let $\eps>0$, $X_i\in\cS$ for $i\in\N$, and $X=\cup_i X_i$. For $i\in\N$, set $Z_i=X_1\cup\ldots\cup X_i$. Then $Z_i\subseteq  Z_{i+1}$ for each $i\in\N$ and $\cup_i Z_i=X$, hence the continuity of $\fg$ from above implies $\fg(X\setminus Z_n)\to 0$ as $n\to\infty$. Thus we have $\fg(X\setminus Z_n)<\eps/2$ if $n$ is large enough.  Note that $Z_n\in\cS$ since $\cS$ is closed under finite unions, so there is a set $Y\in\cA$ with $\fg(Y\triangle Z_n)<\eps/2$. By the submodularity of $\varphi$, these together lead to $\fg(X\triangle Y)<\eps$, implying $X\in\cS$. This concludes the proof of the lemma.
\end{proof}

Now we are ready to state the theorem.

\begin{thm}\label{THM:RCONV-LIMIT-SMOOTH}
Let $\fg$ be a bounded, increasing, submodular setfunction on a sigma-algebra $(J,\cB)$, continuous from above, with $\fg(\emptyset)=0$. Let $(\cP_n\colon n\in\N)$ be a refining sequence of measurable partitions of $J$ into a finite number of sets in $\cB$ such that $\cup_n\cP_n$ generates $\cB$. Then $\fg/{\cP_n} \rightarrowtail \fg$.
\end{thm}
\begin{proof}
Let $\cA_n$ denote the finite set-algebra generated by $\cP_n$. Since the partitions form a refining sequence, $\cA_n\subseteq \cA_{n+1}$ clearly holds. If we set $\cA=\cup_n\cA_n$, then $\cA$ is a countable set-algebra that generates $\cB$. Let $\fg_n=\fg/\cP_n$. We claim that $\fg_n \rightarrowtail \fg$. 

We start with the observation that for any $k\in\N$, we have
\begin{equation}\label{EQ:QUOTIENT-SUB}
Q_k(\fg_n)\subseteq Q_k(\fg).
\end{equation}
This follows immediately from the fact that $\fg_n$ is a quotient of $\fg$. 

Let $\cP_n=(Z_1,\dots,Z_k)$. By Lemma~\ref{LEM:AA-APPROX}, there are sets $X_1,\dots,X_k\in\cA$ such that $\fg(X_i\triangle Z_i)<\eps/(2k^3)$. Our goal is to modify the sets $X_i$ ``a little'' to get a partition. Recall that $\fg$ is assumed to be increasing, submodular with $\fg(\emptyset)=0$; we will use these properties without explicitly referring to them. Note that
$J\setminus(\cup_{i=1}^k X_i) \subseteq \cup_{i=1}^k (Z_i\setminus X_i)$, and hence
\begin{equation*}
\fg\left(J\setminus(\cup_{i=1}^k X_i)\right) \leq \fg\left(\cup_{i=1}^k (Z_i\setminus X_i)\right)\le \sum_{i=1}^k \fg(Z_i\setminus X_i)\le \frac\eps{2k^2}.
\end{equation*}
So, by replacing $X_k$ with $X_k\cup\big(J\setminus(\cup_{i=1}^k X_i)\big)$, we may assume that $\cup_{i=1}^k X_i=J$, at the cost of increasing the error bound from $\eps/(2k^3)$ to $\eps/(2k^2)$. Define $Y_1=X_1$, and for $1<i\leq k$, let $Y_i=X_i\setminus(Y_1\cup\dots\cup Y_{i-1})$. The sets $Y_1,\dots,Y_k$ are clearly disjoint and form a partition $\cY$ of $J$ into sets in $\cA$. By the definition of $\cA$, this means that $\cY$ is a partition of $J$ into sets in $\cA_n$ for every sufficiently large $n$, and so $\fg/\cY\in Q_k(\fg_n)$ for sufficiently large $k$ and $n\ge k$. By the definition of $Y_i$, we have
\begin{equation*}
\fg(Y_i\triangle X_i)=\fg(X_i\setminus Y_i)  = 
\fg\bigl(X_i\cap(X_1\cup\dots\cup X_{i-1})\bigr) \le\sum_{j=1}^{i-1} \fg(X_i\cap X_j).    
\end{equation*}
Since $Z_i\cap Z_j=\emptyset$ for $i\neq j$, we have $X_i\cap X_j\subseteq (X_i\setminus Z_i) \cup(X_j\setminus Z_j)$, this can be further bounded as
\begin{equation*}
\fg(Y_i\triangle X_i)\le i\cdot \fg(X_i\setminus Z_i) + \sum_{j=1}^{i-1} \fg(X_j\setminus Z_j)\leq (2k-1)\cdot\frac{\eps}{2k^2}.    
\end{equation*}
This implies
\begin{equation*}
\fg(Y_i\triangle Z_i)\le \fg(Y_i\triangle X_i) + \fg(X_i\triangle Z_i)
\le (2k-1)\cdot\frac\eps{2k^2}+\frac \eps{2k^2}
= \frac\eps{k}.
\end{equation*}
Therefore, for every $I\subseteq[k]$, we get
\begin{equation*}
|\fg(\cup_{i\in I} Z_i)-\fg(\cup_{i\in I}Y_i)| \le \sum_{i\in I} \fg(Y_i\triangle Z_i)\le \eps.    
\end{equation*}
In other words, the quotients $\fg/\cZ$ and $\fg/\cY$ satisfy $\|\fg/\cZ-\fg/\cY\|\le \eps$. Since $\fg/\cZ$ is an arbitrary quotient in $Q_k(\fg)$ while $\fg/\cY$ is in $Q_k(\fg_n)$, together with~\eqref{EQ:QUOTIENT-SUB} we obtain that
\begin{equation*}
\dH{k}(Q_k(\fg),Q_k(\fg_n))\le\eps.    
\end{equation*}
This completes the proof of the theorem.
\end{proof}

\paragraph{Acknowledgement.} The authors are grateful to Miklós Abért, Boglárka Gehér, András Imolay, Balázs Maga, Tamás Titkos and Dániel Virosztek for helpful discussions. 

Márton Borbényi was supported by the \'{U}NKP-23-3 New National Excellence Program of the Ministry for Culture and Innovation from the source of the National Research, Development and Innovation Fund. László Márton Tóth was supported by the National Research, Development and Innovation Fund -- grant number  KKP-139502. The research was supported by the Lend\"ulet Programme of the Hungarian Academy of Sciences -- grant number LP2021-1/2021, and by Dynasnet European Research Council Synergy project -- grant number ERC-2018-SYG 810115.

%%%%%%%%%%%%%%%%

%%%%%%%%%%%%%%%%%%%%%%%%%%%%%%%%
\bibliographystyle{abbrv}
\bibliography{convergence}

\begin{thebibliography}{10}

\bibitem{bensch2001recurrence}
I.~Benjamini and O.~Schramm.
\newblock Recurrence of distributional limits of finite planar graphs.
\newblock {\em Electronic Journal of Probability}, 6(23):1--13, 2001.

\bibitem{bblt2024graphing}
K.~B{\'e}rczi, M.~Borb{\'e}nyi, L.~Lov{\'a}sz, and L.~M. T{\'o}th.
\newblock Cycle matroids of graphings: From convergence to duality.
\newblock {\em arXiv preprint arXiv2406.08945}, 2024.

\bibitem{bgils2024decomposition}
K.~B\'erczi, B.~Gehér, A.~Imolay, L.~Lovász, and T.~Schwarcz.
\newblock Monotonic decompositions of submodular set functions.
\newblock {\em arXiv preprint arXiv:2406.04728}, 2024.

\bibitem{bckl2013leftandright}
C.~Borgs, J.~T. Chayes, J.~Kahn, and L.~Lov\'asz.
\newblock Left and right convergence of graphs with bounded degree.
\newblock {\em Random Structures and Algorithms}, 42:1--28, 2013.

\bibitem{BCLSV1}
C.~Borgs, J.~T. Chayes, L.~Lov{\'a}sz, V.~T. S{\'o}s, and K.~Vesztergombi.
\newblock Convergent sequences of dense graphs {I}: {S}ubgraph frequencies,
  metric properties and testing.
\newblock {\em Advances in Mathematics}, 219(6):1801--1851, 2008.

\bibitem{BCLSV2}
C.~Borgs, J.~T. Chayes, L.~Lov{\'a}sz, V.~T. S{\'o}s, and K.~Vesztergombi.
\newblock Convergent sequences of dense graphs {II}. {M}ultiway cuts and
  statistical physics.
\newblock {\em Annals of Mathematics}, pages 151--219, 2012.

\bibitem{choquet1954theory}
G.~Choquet.
\newblock Theory of capacities.
\newblock In {\em Annales de l'institut Fourier}, volume~5, pages 131--295,
  1954.

\bibitem{Denn}
D.~Denneberg.
\newblock {\em Non-additive measure and integral}, volume~27 of {\em Theory and
  Decision Library. Series B: Mathematical and Statistical Methods}.
\newblock Kluwer Academic Publishers Group, Dordrecht, 1994.

\bibitem{kardovs2017first}
F.~Kardoš, D.~Král’, A.~Liebenau, and L.~Mach.
\newblock First order convergence of matroids.
\newblock {\em European Journal of Combinatorics}, 59:150--168, 2017.

\bibitem{kechris2012classical}
A.~Kechris.
\newblock {\em Classical descriptive set theory}, volume 156.
\newblock Springer Science \& Business Media, 2012.

\bibitem{lovasz2012large}
L.~Lov{\'a}sz.
\newblock {\em Large networks and graph limits}, volume~60.
\newblock American Mathematical Society, 2012.

\bibitem{lovasz2023matroid}
L.~Lov{\'a}sz.
\newblock The matroid of a graphing.
\newblock {\em arXiv preprint arXiv:2311.03868}, 2023.

\bibitem{lovasz2023submodular}
L.~Lov{\'a}sz.
\newblock Submodular setfunctions on sigma-algebras.
\newblock {\em arXiv preprint arXiv:2302.04704}, 2023.

\bibitem{oxley2011matroid}
J.~Oxley.
\newblock {\em Matroid Theory}, volume~21 of {\em Oxford Graduate Texts in
  Mathematics}.
\newblock Oxford University Press, Oxford, second edition, 2011.

\bibitem{Schrijver}
A.~Schrijver.
\newblock {\em Combinatorial Optimization: Polyhedra and Efficiency}.
\newblock Springer, 2003.

\bibitem{tutte1958homotopy}
W.~T. Tutte.
\newblock A homotopy theorem for matroids {I}, {II}.
\newblock {\em Transactions of the American Mathematical Society},
  88(1):144--172, 1958.

\bibitem{Welsh}
D.~Welsh.
\newblock {\em Matroid Theory}, volume~8 of {\em Londom Mathematical Society
  Monographs}.
\newblock Academic Press, 1976.

\end{thebibliography}

\end{document}